\documentclass[11pt, reqno]{amsart}
\usepackage{xcolor}
\usepackage{mathtools}
\mathtoolsset{showonlyrefs}
\usepackage{etoolbox}
\usepackage{blindtext}
\usepackage{latexsym}
\usepackage[pagewise]{lineno}
\usepackage[utf8]{inputenc}
\usepackage{amscd}
\usepackage{exscale}
\usepackage{amsmath}
\usepackage{amssymb}
\usepackage{amsfonts}
\usepackage{mathrsfs}
\usepackage{amsbsy}
\usepackage{relsize}

\usepackage[colorlinks, linkcolor=purple, citecolor=blue, urlcolor=blue, hypertexnames=false]{hyperref}
\providecommand{\keywords}[1]
{
  \small	
  \textbf{\textit{Keywords---}} #1
}
\usepackage{esint} 
\usepackage{amssymb} 
\usepackage{stmaryrd}
\usepackage{cite} 
\calclayout

\newtheorem{theorem}{Theorem}[section]
\newtheorem{lemma}{Lemma}[section]
\newtheorem{proposition}{Proposition}[section]

\newtheorem{definition}{Definition}[section]
\newtheorem{remark}{Remark}[section]

\newcommand{\R}{\mathbb R}

\newcommand{\ba}{\mathbf a}
\newcommand{\baa}{\mathbf A}

\def\d{\displaystyle}

\everymath{\displaystyle}
\usepackage{indentfirst}
\usepackage{hyperref}
\usepackage{listings}
\usepackage{enumerate}

\numberwithin{equation}{section}

\usepackage{lipsum} 
\usepackage{tikz,xcolor}

\definecolor{lime}{HTML}{A6CE39}
\DeclareRobustCommand{\orcidicon}{
	\begin{tikzpicture}
	\draw[lime, fill=lime] (0,0) 
	circle [radius=0.16] 
	node[white] {{\fontfamily{qag}\selectfont \tiny ID}};
	\draw[white, fill=white] (-0.0625,0.095) 
	circle [radius=0.007];
	\end{tikzpicture}
	\hspace{-2mm}
}
\foreach \x in {A, B}{\expandafter\xdef\csname orcid\x\endcsname{\noexpand\href{https://orcid.org/\csname orcidauthor\x\endcsname}
			{\noexpand\orcidicon}}
}

\title[Time-Dependent Damped NLS]{Energy scattering for the unsteady damped nonlinear Schr\"odinger equation}

\author[Makram Hamouda]{Makram Hamouda\orcidA{}}
\author[Mohamed Majdoub]{Mohamed Majdoub\orcidB{}}

\address[M. Hamouda]{Department of Basic Sciences, Deanship of Preparatory Year and Supporting Studies, Imam Abdulrahman Bin Faisal University, P. O. Box 1982, Dammam, Saudi Arabia.}
\email{\sl \color{blue}{mmhamouda@iau.edu.sa }}
\address[M. Majdoub]{Department of Mathematics, College of Science, Imam Abdulrahman Bin Faisal University, P. O. Box 1982, Dammam, Saudi Arabia}
\address[M. Majdoub]{Basic and Applied Scientific Research Center, Imam Abdulrahman Bin Faisal University, P.O. Box 1982, 31441, Dammam, Saudi Arabia}
\email{\sl \color{blue}{mmajdoub@iau.edu.sa}}
\email{\sl \color{blue}{med.majdoub@gmail.com}}

\begin{document}

\allowdisplaybreaks

	\begin{abstract}
		We investigate the large time behavior of the solutions to the nonlinear focusing Schr\"odinger equation with a time-dependent damping in the energy sub-critical regime. Under  non classical assumptions on the unsteady damping term, we prove some scattering results in the energy space. 
	\end{abstract}
	
	\maketitle

	\section{Introduction}
	\label{S1}

We consider in this article the nonlinear  Schr\"odinger (NLS) equation with an unsteady damping term:
\begin{equation}
\left\{ \begin{array}{ll}
     \label{main}
{\rm i}\partial_t u+\Delta u +{\rm i} {\mathbf a}(t) u=\mu \,|u|^{p-1}u, &\quad (t,x)\in (0,\infty)\times\R^N, \\
    u(0,x) = u_0(x), &\quad x \in \R^N, 
\end{array}
\right.
\end{equation}
where $p>1, \ \mu=\pm1$, and ${\mathbf a} \in \mathcal{C} ([0,\infty))$ is a nonnegative continuous function standing for the loss dissipation. 
 
Considering such a type of damping is not only of mathematical interests but it is also related to concrete real-life applications. For example, the unsteady damping is useful when we deal with Bose-Einstein condensate (BEC) models, or in quantum mechanics when the viscosity is a function of time, thus, giving rise to unsteady friction forces.  

The well-posedness and the scattering for \eqref{main} in the undamped case ($\ba (t) \equiv 0$) are mostly well understood, see e.g. \cite{Cazenave, Tao}. Therefore, we will focus in the reminder of this introduction to make the state of art on the damped regime.

Before getting into the long-time dynamics of \eqref{main}, we give an overview on the known results about the well-posedness in the energy space $H^1(\R^N)$.  For the constant damping case, namely  $\ba (t)=\ba>0$, we have global existence in the defocusing regime ($\mu=1$),  no matter the size of the damping, and for   $1<p<1+\frac{4}{(N-2)_+}$\footnote{Hereafter we use the notation $\kappa_+ :=\max(\kappa,0)$ with the convention $0^{-1}=\infty$.}; see \cite{Cazenave}. However, in the focusing regime ($\mu=-1$), the global existence holds true in the inter-critical case ($1+\frac{4}{N}\leq p<1+\frac{4}{(N-2)_+}$) for a sufficiently large damping $\ba$, see \cite{OhTo}. Note that in the $L^2$ sub-critical regime ($p<1+\frac{4}{N}$), the global existence is guaranteed without any restriction on the damping; see \cite[Proposition B.1]{HM-Damped NLS} which also covers the time-dependent damping. For the latter case, it is proven that a small damping is still leading to  blow-up  as it is the case for the undamped (or even the uniformly damped) NLS problem. However, when the size of the damping is beyond a certain {\it critical value}, the solution exists globally in time under some assumptions on the initial data and  the exponent $p$; see  \cite{HM-Damped NLS} and the references therein.   

Let us now make a short presentation on what it is known so far in the literature about the scattering for \eqref{main} in the whole range $1<p< 1+\frac{4}{(N-2)_+}$. For the constant damping case, the scattering for \eqref{main} is shown in \cite{VDD, In} for both focusing and defocusing regimes under suitable conditions. To the best of our knowledge, the asymptotic behavior of global solutions to \eqref{main} is not sufficiently studied in the literature for the time-dependent damping. Following the approach developed in \cite{Cazenave-21, Cazenave-17, Cazenave-18}, Bamri and Tayachi \cite{Tayachi-23} extend in an elegant way several global existence and scattering results to \eqref{main} for oscillating initial data and a damping function satisfying $t \,\ba(t) \underset{t \to \infty}{\sim} \gamma >0$.

Based on the global existence result in \cite{HM-Damped NLS}, we aim to study the scattering of  solutions to \eqref{main} in the sense of Definition \ref{Def-Scat} below. Assuming  $\underline{\ba}>0$ (see \eqref{ais} below), we prove the scattering in the focusing inter-critical regime. 

It is worth noticing that the aforementioned assumption on the damping term is somehow natural in view of the class of the initial data that we are considering. Indeed, we believe that there is a balance between the conditions on the damping and the initial data. In this context, the  damping $t \,\ba(t) \underset{t \to \infty}{\sim} \gamma >0$ considered in \cite{Tayachi-23} requires a  class of oscillating initial data. 

The present work is intended to extend the results in \cite{In}, where the constant damping  is investigated, to the time-dependent damping case. In fact,  our main target is to show the scattering for \eqref{main} for $H^1-$global solutions assuming \eqref{int-a} below. Note that our  objectives and approaches are  different from the ones in \cite{Tayachi-23}.

In the reminder of the article, we will only be concerned with equation \eqref{main} in the focusing regime which corresponds to $\mu =-1$.

The outline of the article is as follows. In Section \ref{SS1}, we state the main results together with some remarks after giving sense to the meaning of the scattering. Then, in Section \ref{S2}, we start by recalling some standard identities related to the solutions of \eqref{main} followed by several useful preliminary results. Finally, in Section \ref{S3}, we give the proofs of our main results.

\section{Main results}
\label{SS1}

Before we state our main results, we first introduce the following notations:
\begin{gather}
\label{A}
\baa(t)=\int_0^t\,\ba(s)ds,\\
\label{ais}
\underline{\ba}=\inf_{t>0}\,\left(\frac{\baa(t)}{t}\right).
\end{gather}
Obviously, one can see that $\underline{\ba}>0$   yields
\begin{equation}
\label{int-a}
\int_0^\infty \ba(s) ds=\infty.
\end{equation}
The condition \eqref{int-a} is necessary for obtaining the global existence to \eqref{main}  as it was claimed in \cite{HM-Damped NLS}. A natural question to ask is then the long-time behavior of global solutions, that is the scattering. In other words, when and how every global solution of \eqref{main} approaches the solution to the associated free equation in the energy space as $t\to  \infty$. 
\begin{definition}
\label{Def-Scat}
We say that a global solution $u$ to \eqref{main} scatters forward in time in $H^1(\R^N)$ if there exists $u^+\in H^1(\R^N)$ such that
\begin{equation}
\label{Scatt-def}
\displaystyle\lim_{t\to\infty}\,\|{\rm e}^{\baa(t)}{\rm e}^{-{\rm i}t\Delta}\,u(t)-u^{+}\|_{H^1}=0,
\end{equation}
where $\baa(t)$ is given by \eqref{A}.
\end{definition}

\begin{remark}
    {\rm One can intuitively see that the scattering in the sense of \eqref{Scatt-def} gives that ${\rm e}^{\baa(t)}\,\|\nabla u(t)\|_{L^2}$ is bounded in time. This observation is somewhat in relationship with the hypothesis \eqref{Bound-Scatt} below.}
\end{remark}

Let us recall the well-posedness theory in the energy space $H^1(\R^N)$ for the undamped NLS corresponding to \eqref{main} ($\ba(t)=0$). 
\begin{theorem}(\cite[Corlloray 4.3.4, p. 93]{Cazenave})
\label{LWP-NLS}
Let $N\geq 1$ and $\ba(t)\equiv 0$. Assume that
$1<p<1+\frac{4}{(N-2)_+}$. 
If $u_0\in H^1(\R^N)$ then there exist $T_{max}=T_{max}(\|u_0\|_{H^1}, N, p)>0$ and a unique maximal solution $u$ of \eqref{main}  such that 
\begin{equation}
\label{u-space}
u\in C((-T_{max}, T_{max}); H^1(\R^N))\cap L^q_{loc}((-T_{max}, T_{max}); W^{1,r}(\R^N)),
\end{equation}
for any admissible pair $(q,r)$ in the sense of Definition \ref{Admis} below. Moreover, we have the following blow-up criterion:
\begin{equation}
\label{Blow-Crit}
\text{if} \quad T_{max}<\infty \quad \text{then} \quad\displaystyle\lim_{t\to T_{max}}\|\nabla u(t)\|_{L^2(\R^N)}=\infty. 
\end{equation}
\end{theorem}
\begin{remark}
\label{Rem2.2}
~{\rm \begin{itemize}
    \item[$(i)$] Among many other interesting results, the global existence and scattering for the critical value $\displaystyle p=1+\frac{4}{N-2}$ in the defocusing case is studied in \cite{CKSTT-08} for $N= 3$. 
    \item[$(ii)$] Using a change of unknowns $v(t,x)={\rm e}^{\baa(t)}u(t,x)$, where $\baa(t)$ is given by \eqref{A}, and the fact that $t \mapsto {\rm e}^{(1-p)\baa(t)}$ is bounded, Theorem \ref{LWP-NLS} holds true for \eqref{main-bis} below,  and consequently for the damped NLS \eqref{main}.
\end{itemize}
}
\end{remark}
The following scattering results, Theorems \ref{scat1}--\ref{scat3-bis}, can be seen as improvements of \cite{VDD,In}.
\begin{theorem}
\label{scat1}
Let $\d 1<p\leq 1+\frac{4}{N}$, $u_0 \in H^1(\R^N)$ and $u$ be the global solution to \eqref{main}. Assuming $\underline{\ba}>0$, then the solution $u$ scatters in $H^1(\R^N)$.
\end{theorem}

\begin{remark}
\label{Rem2.3}
~\begin{itemize}{\rm
        \item[$(i)$] Note that for $\d 1<p<1+\frac{4}{N}$, the global existence is proven in \cite[Proposition B.1]{HM-Damped NLS}.
        \item[$(ii)$] For $\d p=1+\frac{4}{N}$, the global existence is shown in \cite[Theorem 1.6]{HM-Damped NLS} under the assumption that $\underline{\ba}$, defined by \eqref{ais}, is sufficiently large.
        \item[$(iii)$] Unlike the undamped case, where there is no scattering for the range 
        $1<p \le 1+\frac{2}{N}$; see \cite{Barab,  Strauss2, Tao-2004, Tsut-Yaj}, the scattering of \eqref{main} in the aforementioned range holds under the assumption $\underline{\ba}>0$. 
       \item[$(iv)$] We believe that the scattering result in Theorem \ref{scat1} holds true for a damping function which is only nonnegative for large time.}
    \end{itemize}
\end{remark}

\begin{theorem}
\label{scat3}
Let $\d 1+\frac{4}{N} < p<1+\frac{4}{(N-2)_+}$ and $u_0 \in H^1(\R^N)$. Suppose that $\underline{\ba}>0$. Then, there exists $\varepsilon_0=\varepsilon_0(p,N)>0$ such that if 
\begin{equation}
\label{small-new}
\| \nabla u_0\|_{L^2}\|u_0\|_{L^2}^{\frac{2+N-(N-2)p}{N(p-1)-4}} < \varepsilon_0,
\end{equation}
 the solution to \eqref{main} exists globally and satisfies
\begin{equation}
\label{Decay-u}
\|\nabla u(t)\|_{L^2(\R^N)}\leq C(\varepsilon_0)\, e^{-\baa(t)},\quad \forall \ t\geq 0,
\end{equation}
where $\baa(t)$ is given by \eqref{A}. 
In particular, the solution $u$ scatters in $H^1(\R^N)$.
\end{theorem}
\begin{remark}
~\begin{itemize}{\rm
        \item[$(i)$] It is well known that for the classical focusing intercritical NLS, the global existence and scattering are guaranteed under the smallness assumption on the initial data in $H^1$, see \cite[Theorem 6.2.1, p. 165]{Cazenave} and \cite[Remark 7.9.6, p. 253]{Cazenave}. 
        \item[$(ii)$] The situation in the damped case, namely \eqref{main}, is different. Indeed, for the global existence, we have a balance between choosing small initial data and large damping. See \cite[Theorem 1.6]{HM-Damped NLS} for the global existence in the case of large damping.
        \item[$(iii)$] The proof of global existence in Theorem \ref{scat3} relies solely on the new smallness condition introduced in \eqref{small-new}. Interestingly, the new condition \eqref{small-new} is weaker than the classical smallness hypothesis in $H^1$. In particular, one can consider, for example, initial data such that $\|u_0\|_{L^2}\gg 1$ and  $\|\nabla u_0\|_{L^2}\ll 1$.
        \item[$(iv)$]  The scattering part in Theorem \ref{scat3} is ensured by the smallness of $u_0$ in the sense of \eqref{small-new} together with $\underline{\ba}>0$. Notably, the condition $\underline{\ba}>0$ is only required for the scattering in view of Lemma \ref{L2.5} below.}
        \end{itemize}
\end{remark}

\begin{theorem}
\label{scat3-bis}
Let $\d 1+\frac{4}{N} < p<1+\frac{4}{(N-2)_+}$ and $u_0 \in H^1(\R^N)$.  Let $u$ be the global solution to \eqref{main}. Assume that $\underline{\ba}>0$ holds true. Furthermore, suppose that 
\begin{equation}
    \label{hyp-sup} \sup_{t \ge 0} \|u(t)\|_{L^{p+1}} < \infty.
\end{equation}
Then, for all $\varepsilon >0$, there exists $T>0$ such that
\begin{equation}
    \label{exp-scat} e^{-\underline{\ba}\, T} \|  u (T)\|_{H^1} < \varepsilon,
\end{equation}
where $\underline{\ba}>0$ is given by \eqref{ais}. In particular, the solution $u$ scatters in $H^1(\R^N)$.
\end{theorem}

\begin{remark}
    {\rm Obviously, in view of Sobolev embedding, the assumption \eqref{hyp-sup} is fulfilled once we have 
    \begin{equation}
        \label{hyp-H1} \sup_{t \ge 0} \|u(t)\|_{H^1} < \infty.
    \end{equation}
    Note that the hypothesis \eqref{hyp-H1} appears e.g. in \cite{Bourgain-1996, Tao-2004}.} 
\end{remark}
\begin{remark}
    {\rm As stated in the beginning of this section, all the above energy scattering results require that $\underline{\ba}>0$, where $\underline{\ba}$ is given by \eqref{ais}. It is then natural to investigate the case $\underline{\ba}=0$ for both global existence and scattering goals. In this direction and for oscillating initial data, the recent article \cite{Tayachi-23} treats a class of dampings behaving like  $\ba(t)\sim \frac{\gamma}{1+t},\, t\gg  1, \gamma>0$.  Clearly, the aforementioned damping fulfills $\underline{\ba}=0$. In view of the above observations, it is interesting to study the following equation:
    \begin{equation}
        \label{Main-New}
       {\rm i}\partial_t u+\Delta u + \frac{{\rm i}a}{(1+t)^{\theta}}\, u=-|u|^{p-1}u,\;\;\; a,\theta\geq 0.
    \end{equation}
  The global existence and scattering for \eqref{Main-New} is under investigation  \cite{HM-New}.  }
    \end{remark}

\section{Useful tools \& Auxiliary results}
\label{S2}
	
To state our main results in a clear way, we define the following quantities:
\begin{gather}
\label{Mass}
{\mathbf M}(u(t))=\int_{\R^N}\,|u(t,x)|^2\,dx,
\vspace{.3cm}\\
\label{Ener}
{\mathbf E}(u(t))=\|\nabla u(t)\|_{L^2}^2-\frac{2}{p+1}\int_{\R^N}\, |u(t,x)|^{p+1}\,dx,\vspace{.3cm}\\
\label{Iu}
{\mathbf I}(u(t))=\|\nabla u(t)\|_{L^2}^2-\int_{\R^N}\, |u(t,x)|^{p+1}\,dx.\vspace{.3cm}\\
\end{gather}

Some crucial relationships between the above quantities are summarized in the following.
\begin{proposition}
\label{Rela}
Let $u$ be a sufficiently smooth solution of \eqref{main} on $0\leq t\leq T$. Then, we have
\begin{gather}
\label{M-Id}
{\mathbf M}(u(t))={\rm e}^{-2\baa(t)}{\mathbf M}(u_0),\vspace{.3cm}\\
\label{E-Id}
\frac{d}{dt}{\mathbf E}(u(t))=-2\ba(t) {\mathbf I}(u(t)).\vspace{.3cm}\\
\end{gather}
\end{proposition}

\begin{proof}[{Proof of Proposition \ref{Rela}}]
Although we deal here with a time-dependent damping, the proof mimics the same steps performed in \cite[Lemma 1]{Tsut1} where the constant damping is considered. The details are hence omitted.
\end{proof}

In order to facilitate the reading of the paper, we collect here some mathematical tools. First, we state a Gr\"onwall type inequality.
\begin{lemma}
    \label{gronwall}
    Let $f: [0,\infty) \to [0,\infty)$ be a nonnegative continuous function satisfying 
    \begin{equation}
        \label{gronw-ineq}
        f(t) \le C + g(t) (f(t))^{\beta} + \int_0^t h(s)(f(s))^{\beta} ds, \quad \forall \ t \ge 0,
    \end{equation}
    where $C$ is a positive constant, $0<\beta\le 1$ and $g(t), h(t)$ are two nonnegative continuous functions on $[0,\infty)$. Furthermore, assume that 
    \begin{equation}
        \label{bound-h}
        t\mapsto \int_0^t h(s) ds \ \text{is bounded on}\  [0,\infty),
    \end{equation}
    and either
    \begin{equation}
        \label{limit-0}
       \beta=1, \quad \text{and}\quad  \lim_{t \to \infty} g(t)=0,
    \end{equation}
    or
    \begin{equation}
        \label{beta-g}
       \beta<1, \quad \text{and}\quad  g \ \text{is bounded on}\ [0,\infty).
    \end{equation}
    Then, the function $f$ is bounded on $[0,\infty)$.
\end{lemma}
\begin{proof}
First, we consider the case $\beta=1$. Thanks to \eqref{limit-0}, there exists $t_0>0$ such that $g(t)\leq 1/2$ for all $t>t_0$. Hence, the inequality \eqref{gronw-ineq} translates to
\begin{equation*}
    \label{beta=1}
    f(t)\leq 2C+2 \int_0^t h(s)f(s)ds, \quad \forall \ t > t_0.
\end{equation*}
Consequently, since $f$ is continuous and nonnegative, we have
\begin{equation}
    \label{beta=11}
    f(t)\leq 2C+\sup_{0\leq t\leq t_0} f(t)+2 \int_0^t h(s)f(s)ds, \quad \forall \ t \ge 0.
\end{equation}
Applying Gr\"onwall's inequality to \eqref{beta=11}, we obtain
\begin{equation*}
        \label{gronw-conc-beta1}
        f(t) \le \left(2C+\sup_{0\leq t\leq t_0} f(t)\right) \exp\left(2 \int_0^t h(s) ds\right), \quad \forall \ t \ge 0.
    \end{equation*}
In view of \eqref{bound-h}, the above estimate implies that $f$ is bounded on $[0,\infty)$.

Now, for the case $0<\beta<1$, we rewrite \eqref{gronw-ineq} using some elementary inequalities. Indeed, for the second term in the RHS of \eqref{gronw-ineq}, we use the following Young inequality
\begin{equation*}
ab^\beta\leq \frac{1}{2}b+(1-\beta)\left(2\beta\right)^{\frac{\beta}{1-\beta}} a^{\frac{1}{1-\beta}},\quad a,\, b\geq 0.
\end{equation*}
For the third term in the RHS of \eqref{gronw-ineq}, we employ the simple convex inequality 
\begin{equation*}
x^\beta\leq \beta\,x+1-\beta,\quad x\geq 0. 
\end{equation*}
Therefore, we obtain, for all $t\ge 0$,
\begin{equation}
\label{Gronw-1}
f(t)\leq 2C+2(1-\beta)\left(2\beta\right)^{\frac{\beta}{1-\beta}}\|g\|_{\infty}+2(1-\beta)\int_0^t h(s)ds+2\beta\int_0^t h(s)f(s) ds.
\end{equation}
Applying the standard Gr\"onwall inequality to \eqref{Gronw-1} yields, for all $t\ge 0$,
    \begin{equation}
        \label{gronw-conc}
        f(t) \le \left(2C+2(1-\beta)\left(2\beta\right)^{\frac{\beta}{1-\beta}}\|g\|_{\infty}+2(1-\beta)\int_0^t h(s)ds\right) \exp\left(2\beta \int_0^t h(s) ds\right).
    \end{equation}
Thanks to \eqref{bound-h}, the estimate \eqref{gronw-conc} yields the boundedness of $f$ on $[0,\infty)$.

    The proof of Lemma \ref{gronwall} is now  complete.
\end{proof}

{Second, the following continuity argument (or bootstrap argument) will also be useful for our scattering purpose. See \cite[Lemma 2.11]{DKM} for a similar statement.
\begin{lemma}\cite[Lemma 3.7, p. 437]{Strauss}\\
\label{boots}
Let $\mathbf{I}\subset\R$ be a time interval, and $\mathbf{X} : \mathbf{I}\to [0,\infty)$ be a continuous function satisfying, for every $t\in \mathbf{I}$,
\begin{equation}
		\label{boots1}
		    	\mathbf{X}(t) \leq a + b [\mathbf{X}(t)]^\theta,
		\end{equation}
	where $a,b>0$ and $\theta>1$ are constants. Assume that, for some $t_0\in \mathbf{I}$,
		\begin{equation}
		\label{boots2}
	\mathbf{X}(t_0)\leq a, \quad a\,b^{\frac{1}{\theta-1}} <(\theta-1)\,\theta^{\frac{\theta}{1-\theta}}.
				\end{equation}
		Then, for every $ t\in \mathbf{I}$, we have
		\begin{equation}
		\label{boots3}
		    	\mathbf{X}(t)< \frac{\theta\,a}{\theta-1}.
		\end{equation}
\end{lemma}

Let us recall the following  Gagliardo-Nirenberg type inequality \cite{Far, Genoud2012, pz} which plays a crucial role in our proofs.
\begin{lemma}\label{GNI}
Let $N\geq 1$ and $1<p<1+\frac{4}{(N-2)_+}$. Then the following Gagliardo-Nirenberg inequality holds
\begin{equation}
\int_{\R^N} |u(x)|^{p+1}\,dx
\leq\, K\,\|u\|_{L^2}^{p+1-\sigma} \ \|\nabla u\|_{L^2}^{\sigma},\label{Nirenberg}\end{equation}
where 
\begin{equation}
\label{AB}
\sigma=\frac{N}{2}(p-1).
\end{equation}
\end{lemma}

Next, we recall the Strichartz estimates. But, before that we need the following definition.
	\begin{definition}
 \label{Admis}
	A pair $(q,r)$ is said to be  admissible if
		\[
		\frac{2}{q}+\frac{N}{r}=\frac{N}{2}, \quad \left\{
		\renewcommand*{\arraystretch}{1.2}
		\begin{array}{ll}
		r \in \left[2,\frac{2N}{N-2}\right] &\text{if } N\geq 3, \\
		r \in [2, \infty) &\text{if } N=2, \\
		r\in [2, \infty] &\text{if } N=1.
		\end{array}
		\right.
		\]	\end{definition}
Recall that the free propagator associated with \eqref{main} is given by
\begin{equation}
\label{U-a}
U_{\ba}(t)={\rm e}^{-\baa(t)}\,{\rm e}^{it\Delta},
\end{equation}
where $\baa(t)$ is given by \eqref{A}. It is then quite standard that the Cauchy problem  \eqref{main} can be written in an integral form (see \cite{Cazenave}):
\begin{equation}
\label{Integ-Eq}
u(t)=U_{\ba}(t)u_0+i \int_0^t\,U_{\ba}(t-\tau)\left(|u(\tau)|^{p-1}u(\tau)\right)\,d\tau.
\end{equation}
	
	\begin{proposition}[Strichartz estimates] \label{prop-strichartz}
		Let $N\geq 3$, $(q,r)$ be an  admissible pair, $2 < \gamma < \frac{2N}{N-2}$ and $\alpha, \beta \in (1,\infty)$ such that $$\displaystyle \frac{1}{\alpha}+\frac{1}{\beta}=N \left(\frac{1}{2}-\frac{1}{\gamma}\right).$$ 
  Let $0 <T \le \infty$. Then, there exists a constant $C>0$ independent of $T$ such that
  \begin{eqnarray*}
      && \left\| \int_0^t U_{\ba}(t-s) F(s) ds \right\|_{L^q(0,T;L^r)} \leq C \|F\|_{L^{{q}'}(0,T;L^{{r}'})},\\
      &&\left\| \int_0^t U_{\ba}(t-s) F(s) ds \right\|_{L^{\infty}(0,T;L^2)} \leq C \|F\|_{L^{{q}'}(0,T;L^{{r}'})},\\
      &&\left\| \int_0^t U_{\ba}(t-s) F(s) ds \right\|_{L^{\alpha}(0,T;L^{\gamma})} \leq C \|F\|_{L^{{\beta}'}(0,T;L^{{\gamma}'})},
  \end{eqnarray*}
  where $U_{\ba}(t)$ is given by \eqref{U-a}. 
  \end{proposition}
\begin{remark}
    \textnormal{Note that the proofs of the above estimates are given in  \cite[Lemmas 5--7]{OhTo} for the case of a constant damping. The extension to a nonnegative time-dependent damping is straightforward by following the same steps.}
\end{remark}

\begin{remark}
    \textnormal{For further readings about Strichartz estimates for dispersive equations, see e.g. \cite{BCD, Cazenave, Tao}.}
    \end{remark}

Let us now define\; 
\begin{equation}
\label{Paramet}
    \theta=\frac{2(p-1)(p+1)}{4-(N-2)(p-1)}, \quad q_0=\frac{4(p+1)}{N(p-1)},\quad r_0=p+1.   
\end{equation}
Clearly $(q_0, r_0)$ is an admissible pair. We state the following proposition  which gives a sufficient condition for the global existence and the scattering for \eqref{main}. 
\begin{proposition}
\label{Oh-To}
Let $N\geq 3$,  $1+\frac{4}{N}\leq p<1+\frac{4}{N-2}$ and $u_0\in H^1(\R^N)$. Assume that $\underline{\ba}>0$. Then, there exist $\varepsilon>0$ (independent of $\ba$ and $u_0$) such that $T^*_{\ba}(u_0)=\infty$ provided that $\|U_{\ba}(\cdot)u_0\|_{L^\theta(0,\infty;~L^{r_0})}~\leq~ \varepsilon$, where $\theta$ is given by \eqref{Paramet}. Moreover, we have 
\begin{equation}
    \label{p-r0-est}
    \|u\|_{L^{\theta}([0,\infty); L^{r_0})} \le 2 \varepsilon.
\end{equation}
In particular, the solution $u$ scatters in $H^1(\R^N)$.
\end{proposition}
\begin{proof}
First, note that the global existence part is proven in \cite[Proposition 3.1]{HM-Damped NLS}, see also \cite[Proposition 3]{OhTo} for the case of a constant damping. Second, for the proof of \eqref{p-r0-est}, we follow the same steps as in the proof of Proposition 3 in \cite{OhTo}. Indeed, it suffices to use the Strichartz estimates in Proposition \ref{prop-strichartz} and the boostrap argument in Lemma \ref{boots} to conclude the proof of \eqref{p-r0-est}. Third, the scattering conclusion follows in a standard way in view of Lemma \ref{L2.4} below for which the assumption \eqref{theta-r-0} is fulfilled thanks to \eqref{p-r0-est}.
\end{proof}

\begin{lemma}
\label{L2.3}
Let $N\geq 3$, $1<p<1+\frac{4}{N-2}$ and $u_0\in H^1(\R^N)$. Suppose that the global solution $u$ of \eqref{main}  satisfies
\begin{equation}
\label{theta-r-0}
\|u\|_{L^{\theta}([0,\infty); L^{r_0})} <\infty.
\end{equation}
Then, we have
\begin{equation}
\label{q0-r0}
\left\|e^{\underline{\ba} \, t}u\right\|_{L^{q_0}([0,\infty); W^{1, r_0})} <\infty,
\end{equation}
where $\underline{\ba}$ is given by \eqref{ais}.
\end{lemma}
\begin{proof}
The proof follows the same lines as the one in \cite[Lemma 2.3]{In} with the necessary modifications for the time-dependent damping case that we will explain in the subsequent. \\
Using \eqref{Integ-Eq}, we have for all $0<\tau<t$,
\begin{equation*}
    e^{\underline{\ba} \, t} u(t)=e^{\underline{\ba} \, t -\baa(t)}{\rm e}^{it\Delta} u(\tau)+i \int_{\tau}^t\,e^{\underline{\ba} \, t -\baa(t-\sigma)}{\rm e}^{i(t-\sigma)\Delta}\left(|u(\sigma)|^{p-1}u(\sigma)\right)\,d\sigma.
\end{equation*}
From the definition of $\underline{\ba}$, given by \eqref{ais}, we easily infer that 
\begin{equation}
    \label{observ}
    \underline{\ba} \, t - \baa(t) \le 0, \quad \underline{\ba} \, t -\baa(t-\sigma) \le \underline{\ba} \, \sigma.
\end{equation}
Employing the Strichartz estimates together with \eqref{observ}, we deduce that for $0<\tau<s$,
\begin{equation}
\label{strich}
\begin{array}{rcl}
   \left\|e^{\underline{\ba} \, t}u\right\|_{L^{q_0}((\tau,s); W^{1, r_0})} &\le& C \|u(\tau)\|_{H^1}+\left\|e^{\underline{\ba} \, \sigma}|u(\sigma)|^{p-1}u(\sigma)\right\|_{L^{q'_0}((\tau,s); W^{1, r'_0})} \\
   &\le& C \|u(\tau)\|_{H^1}+C \|u\|^{p-1}_{L^{\theta}((\tau,s); L^{r_0})}\left\|e^{\underline{\ba} \, t}u\right\|_{L^{q_0}((\tau,s); W^{1, r_0})},
   \end{array}
\end{equation}
where $C$ is a positive constant.\\
Thanks to the hypothesis \eqref{theta-r-0}, one can choose $\tau=\tau_0>0$ large enough such that
\begin{displaymath}
    \|u\|^{p-1}_{L^{\theta}((\tau_0,\infty); L^{r_0})} < \frac{1}{2C}.
\end{displaymath}
Now, for $s>\tau_0$, we have from \eqref{strich} that
\begin{displaymath}
    \left\|e^{\underline{\ba} \, t}u\right\|_{L^{q_0}((\tau_0,s); W^{1, r_0})} \le 2C \|u(\tau_0)\|_{H^1}.
\end{displaymath}
Since $s>\tau_0$ is arbitrary, we obtain that
\begin{displaymath}
    \left\|e^{\underline{\ba} \, t}u\right\|_{L^{q_0}((\tau_0,\infty); W^{1, r_0})} \le 2C \|u(\tau_0)\|_{H^1}.
\end{displaymath}
In view of Theorem \ref{LWP-NLS}, Remark \ref{Rem2.2} $(ii)$ and the fact that $(q_0, r_0)$ is an admissible pair, we have $u \in L_{loc}^{q_0}([0,\infty); W^{1, r_0})$. This completes the proof of Lemma \ref{L2.3}.
\end{proof}
\begin{lemma}
\label{L2.4}
Let $N\geq 3$, $1<p<1+\frac{4}{N-2}$ and $u_0\in H^1(\R^N)$. Suppose that the global solution $u$ of \eqref{main}  satisfies \eqref{theta-r-0}. Then $u$   scatters in $H^1(\R^N)$.
\end{lemma}
\begin{proof}
Let $0<t<s<\infty$. Arguing as in the proof of Lemma \ref{L2.3}, we infer that
\begin{equation*}
\left\|e^{\underline{\ba} \, t }{\rm e}^{-it\Delta} u(t)-e^{\underline{\ba} \, s }{\rm e}^{-is\Delta} u(s)\right\|_{H^1}\lesssim  \|u\|^{p-1}_{L^{\theta}((t,s); L^{r_0})}\left\|e^{\underline{\ba} \, t}u\right\|_{L^{q_0}((t,s); W^{1, r_0})}.
\end{equation*}
Thanks to the hypothesis \eqref{theta-r-0} and the conclusion \eqref{q0-r0} in Lemma \ref{L2.3}, we obtain the existence of $u_+\in H^1(\R^N)$ such that $e^{\underline{\ba} \, t }{\rm e}^{-it\Delta} u(t)$ converges, in $H^1(\R^N)$, to $u_+$ as $t\to\infty$.
\end{proof}

\begin{remark}
    {\rm Note that Lemmas \ref{L2.3} and \ref{L2.4} are valid even for $\underline \ba =0$.}
\end{remark}

\begin{lemma}
\label{L2.5}
Let $N\geq 3$, $1<p<1+\frac{4}{N-2}$ and $u_0\in H^1(\R^N)$. Assuming that $\underline \ba >0$. Moreover, suppose that the global solution $u$ of \eqref{main}  satisfies
\begin{equation}
\label{Bound-Scatt}
\sup_{t\geq 0}e^{\underline{\ba} \, t }\,\left\|\nabla u(t)\right\|_{L^2} <\infty.
\end{equation}
Then $u$   scatters in $H^1(\R^N)$.
\end{lemma}
\begin{proof}
Employing the mass identity \eqref{M-Id}, the hypothesis \eqref{Bound-Scatt} and the definition of $\underline{\ba}$, \eqref{ais}, we obtain
\begin{equation}
\label{Scatt-1}
e^{\underline{\ba}\,t} \|u(t)\|_{H^1}\lesssim e^{\underline{\ba}\,t} \|u(t)\|_{L^2}+e^{\underline{\ba}\,t} \|\nabla u(t)\|_{L^2}\leq C<\infty.
\end{equation}
Therefore, by the Sobolev embedding $H^1\hookrightarrow L^{r_0}$ and \eqref{Scatt-1}, we deduce that
\begin{equation*}
\|u\|_{L^{\theta}((0,\infty); L^{r_0})}=\|e^{-\underline{\ba}\,t} \left(e^{\underline{\ba}\,t}u\right)\|_{L^{\theta}((0,\infty); L^{r_0})}\lesssim \|e^{-\underline{\ba}\,t}\|_{L^{\theta}((0,\infty))}<\infty,
\end{equation*}
where we used the fact that $\underline{\ba}>0$. The conclusion follows from Lemma \ref{L2.4}.
\end{proof}
\begin{lemma}
\label{L2.9}
Let $N\geq 3$, $1<p<1+\frac{4}{N-2}$, $u_0\in H^1(\R^N)$ and $u$  be the global solution  of \eqref{main}. Assuming \eqref{int-a}, we have
\begin{equation}
\label{liminf-}
\displaystyle\liminf_{t\to\infty}\,{\mathbf I}(u(t))\leq 0,
\end{equation}
where ${\mathbf I}(u(t))$ is defined by \eqref{Iu}.
\end{lemma}
\begin{proof}
We argue by contradiction. Suppose that
\[
\displaystyle\liminf_{t\to\infty}\,{\mathbf I}(u(t))> 0.
\]
Then, there exist $\delta,~T_0>0$ such that 
\begin{equation}
\label{I-delta}
{\mathbf I}(u(t))\geq \delta,\quad \forall \ t\geq T_0.
\end{equation}
The above inequality together with \eqref{E-Id} imply that
\[
\frac{d}{dt}{\mathbf E}(u(t))\leq -2\delta \ba(t),\quad \forall \ t\geq T_0.
\]
Integrating the last inequality yields
\[
{\mathbf E}(u(t))\leq {\mathbf E}(u(T_0))-2\delta\int_{T_0}^t \ba(s) ds, \quad \forall \ t\geq T_0.
\]
Owing to \eqref{int-a}, there exists $T>T_0$ such that 
\begin{equation}
\label{E-negative}
 {\mathbf E}(u(t))<0,\quad \forall \ t\geq T.  
\end{equation}
Finally, combining \eqref{Ener}, \eqref{Iu}, \eqref{I-delta} and \eqref{E-negative} yields
\[
\delta\leq {\mathbf I}(u(t))\leq {\mathbf E}(u(t))<0, \quad \forall \ t\geq T.
\]
This obviously leads to a contradiction.
\end{proof}

\bigskip 

Using a change of unknowns $v(t,x)={\rm e}^{\baa(t)}u(t,x)$, where $\baa(t)$ is given by \eqref{A}, the equation \eqref{main} can be written as
\begin{equation}\label{main-bis}
\left\{
\begin{array}{l}
{\rm i}\partial_t v+\Delta v =-{\rm e}^{(1-p)\baa(t)}|v|^{p-1}v,\\
v(0,x)=u_0(x).
\end{array}
\right.
\end{equation}
Obviously, the blow-up or the global existence of $v$ implies those of $u$, and vice-versa.

The local well-posedness of \eqref{main-bis} follows as for the classical NLS since $t \mapsto {\rm e}^{(1-p)\baa(t)}$ is bounded, see e.g. \cite{Cazenave}. Furthermore, the mass conservation  holds true for $v$, namely
\begin{equation}\label{mass-v}
    \mathbf{M}(v(t))=\mathbf{M}(u_0).
\end{equation}
Now, we define the Hamiltonian of $v$ by
\begin{equation}\label{Hamil-v}
    \mathbf{H}(v(t)):=\|\nabla v(t)\|_{L^2}^2-\frac{2}{p+1}{\rm e}^{(1-p)\baa(t)}\int_{\R^N}\, |v(t,x)|^{p+1}\,dx.
\end{equation}
A direct computation gives
\begin{equation}\label{Hamil-v-dt}
    \frac{d}{dt}\mathbf{H}(v(t))=\frac{2(p-1)}{p+1} \,\ba(t){\rm e}^{(1-p)\baa(t)}\int_{\R^N}\, |v(t,x)|^{p+1}\,dx.
\end{equation}
Integrating the above equation in time yields
\begin{equation}\label{Hamil-v-id}
    \mathbf{H}(v(t))=\mathbf{E}(u_0)+\frac{2(p-1)}{p+1}\int_0^t \ba(s){\rm e}^{(1-p)\baa(s)}\left(\int_{\R^N}\, |v(s,x)|^{p+1} dx\right)ds.
\end{equation}
Let us  also define the Hamiltonian of  $u$ by
\begin{equation}
    \label{tildeH}
{\bf \mathcal{H}}(u(t))=e^{2\baa(t)}{\mathbf E}(u(t))-\frac{2(p-1)}{p+1}\int_0^t \ba(s) e^{2\baa(s)}\|u(s)\|_{L^{p+1}}^{p+1} ds.
\end{equation}
Using \eqref{E-Id}, one can easily see that
\begin{equation}
\label{Ham-u}
\frac{d}{dt}{\bf \mathcal{H}}(u(t))=0.
\end{equation}
\section{Proofs of main results}
\label{S3}
\subsection{Proof of Theorem \ref{scat1}}
From \eqref{tildeH}-\eqref{Ham-u}, we have
\begin{equation}
\label{H-Id}
{\mathbf E}(u_0)=e^{2\baa(t)}{\mathbf E}(u(t))-\frac{2(p-1)}{p+1}\int_0^t \ba(s) e^{2\baa(s)}\|u(s)\|_{L^{p+1}}^{p+1} ds,
\end{equation}
where $\baa(t)$ and ${\mathbf E}(u(t))$ are given by \eqref{A} and \eqref{Ener}, respectively.\\
Using the mass identity \eqref{M-Id} and the Gagliardo-Nirenberg inequality \eqref{Nirenberg}, the equation \eqref{H-Id} gives
\begin{equation}
\label{Pre-Gron}
\begin{array}{rcl}
{\mathbf E}(u_0)&\geq& e^{2\baa(t)}\|\nabla u(t)\|_{L^2}^2-\frac{2K}{p+1}e^{2\baa(t)}\|u(t)\|_{L^2}^{p+1-\sigma} \ \|\nabla u(t)\|_{L^2}^{\sigma}\\
&-&\frac{2K(p-1)}{p+1}\int_0^t\, \ba(s) e^{2\baa(s)} \|u(s)\|_{L^2}^{p+1-\sigma} \ \|\nabla u(s)\|_{L^2}^{\sigma} ds,
\end{array}
\end{equation}
where $\sigma$ is given by \eqref{AB}.\\
Thanks to \eqref{M-Id}, the inequality \eqref{Pre-Gron} yields
\begin{equation*}
        f(t) \le C + g(t) (f(t))^{\beta} + \int_0^t h(s)(f(s))^{\beta} ds, \quad \forall \ t \ge 0,
    \end{equation*}
    where
    \begin{eqnarray*}
        f(t)&=&e^{2\baa(t)}\|\nabla u(t)\|_{L^2}^2,\\
        g(t)&=&\frac{2K}{p+1}\|u_0\|_{L^2}^{p+1-\sigma} e^{-(p-1)\baa(t)},\\
        h(t)&=&\frac{2K(p-1)}{p+1}\|u_0\|_{L^2}^{p+1-\sigma}\,\ba(t) e^{-(p-1)\baa(t)},\\
        C&=&{\mathbf E}(u_0),\\
        \beta&=&\frac{\sigma}{2}.
    \end{eqnarray*}
Since $p\leq 1+\frac{4}{N}$ and the damping $\ba(t)$ satisfies \eqref{int-a}\footnote{In fact, the assumption \eqref{int-a} is only needed for the mass-critical case $p=1+\frac{4}{N}$.}, it is easy to check that all the hypotheses of Lemma \ref{gronwall} are fulfilled. Hence, the aforementioned lemma implies that 
\begin{equation*}
    \sup_{t\geq 0}\, e^{\baa(t)}\|\nabla u(t)\|_{L^2}<\infty.
\end{equation*}
Owing to \eqref{observ}, we deduce \eqref{Bound-Scatt}. Finally, by applying Lemma \ref{L2.5}, the proof of Theorem \ref{scat1} is achieved.
\subsection{Proof of Theorem \ref{scat3}}
First, we define 
\[
X(t)=e^{\baa(t)}\|\nabla u(t)\|_{L^2}.
\]
Then, employing the conservation property \eqref{Ham-u}, the definition of $\mathcal{H}$, given by \eqref{tildeH}, and the Gagliardo-Nirenberg inequality \eqref{Nirenberg}, we infer that
\begin{equation}
\label{H-ineq1}
\begin{split}
E(u_0)=\mathcal{H}(u(t))&\geq (X(t))^2-\frac{2K}{p+1}e^{2\baa(t)}\|u(t)\|_{L^2}^{p+1-\sigma}\|\nabla u(t)\|_{L^2}^{\sigma}\\
&-\frac{2K(p-1)}{p+1}\int_0^t \ba(s) e^{2\baa(s)}\|u(s)\|_{L^2}^{p+1-\sigma}\|\nabla u(s)\|_{L^2}^{\sigma} ds,
\end{split}
\end{equation}
where $\sigma$ is given by \eqref{AB}.
Thanks to \eqref{M-Id}, the inequality \eqref{H-ineq1} gives
\begin{equation}
\label{H-ineq2}
\begin{split}
(X(t))^2&\leq E(u_0)+\frac{2K}{p+1}e^{(1-p)\baa(t)}\|u_0\|_{L^2}^{p+1-\sigma}(X(t))^{\sigma}\\
&+\frac{2K(p-1)}{p+1}\|u_0\|_{L^2}^{p+1-\sigma}\int_0^t \ba(s) e^{(1-p)\baa(s)}(X(s))^{\sigma} ds.
\end{split}
\end{equation}
Define, for $t\geq 0$, 
\[
X^*(t)=\sup_{0\leq s\leq t} X(s).
\]
Then, for $0<\tau<t$, the estimate \eqref{H-ineq2} (written at time $\tau$) gives
\begin{equation}
\label{H-ineq3}
 (X(\tau))^2 \leq E(u_0)+ \frac{2K}{p+1}\|u_0\|_{L^2}^{p+1-\sigma}(X(\tau))^{\sigma}+\frac{2K}{p+1}\|u_0\|_{L^2}^{p+1-\sigma}(X^*(\tau))^{\sigma},
\end{equation}
where we have used $(p-1)\int_0^t \ba(s) e^{(1-p)\baa(s)} ds\leq 1$.\\
Taking the supremum over $0<\tau<t$ in \eqref{H-ineq3} and using the fact that $X^*(\tau)\leq X^*(t)$ together with $E(u_0)\leq \|\nabla u_0\|_{L^2}^2$, we obtain
\begin{equation}
\label{H-ineq4}
(X^*(t))^2 \leq \|\nabla u_0\|_{L^2}^2+ C(X^*(t))^{\sigma},
\end{equation}
where $C=\frac{4K}{p+1}\|u_0\|_{L^2}^{p+1-\sigma}$.\\
Since $\sigma=\frac{N}{2}(p-1)>2$ and using the smallness of $\| \nabla u_0\|_{L^2}\,\|u_0\|_{L^2}^{\frac{2+N-(N-2)p}{N(p-1)-4}} $ which guarantees \eqref{boots2}, Lemma \ref{boots} (with $\theta=\sigma/2$) applied to \eqref{H-ineq4} yields \eqref{Decay-u}. Hence, thanks to the simple observation $\baa (t) \ge \underline{\ba} t$, the estimate \eqref{Decay-u} implies \eqref{Bound-Scatt} which concludes the scattering by application of Lemma \ref{L2.5}.

This completes the proof of Theorem \ref{scat3}.

\subsection{Proof of Theorem \ref{scat3-bis}}
 From the definition \eqref{Iu}, we have
 \[
 \|\nabla u(t)\|_{L^2}^2={\mathbf I}(u(t))+\|u(t)\|_{L^{p+1}}^{p+1}.
 \]
Owing to \eqref{hyp-sup}, we infer that
\[
e^{-2\underline{\ba}\, t}\|\nabla u(t)\|_{L^2}^2\leq e^{-2\underline{\ba}\, t}{\mathbf I}(u(t))+C e^{-2\underline{\ba}\, t},
\]
where $\underline{\ba}>0$ is given by \eqref{ais}.
Taking the limit inferior, as $t\to\infty$, of both sides in the above inequality yields
\begin{equation}
\label{liminf}
\displaystyle\liminf_{t\to\infty}\left(e^{-2\underline{\ba}\, t}\|\nabla u(t)\|_{L^2}^2\right)\leq \liminf_{t\to\infty}\left(e^{-2\underline{\ba}\, t}{\mathbf I}(u(t))\right)+\limsup_{t\to\infty}\left(C e^{-2\underline{\ba}\, t}\right).
\end{equation}
Employing \eqref{liminf-}, we conclude from \eqref{liminf} that
\[
0\leq \liminf_{t\to\infty}\left(e^{-2\underline{\ba}\, t}\|\nabla u(t)\|_{L^2}^2\right)\leq \liminf_{t\to\infty}\left(e^{-2\underline{\ba}\, t}{\mathbf I}(u(t))\right)\leq 0.
\]
Consequently, we obtain
\begin{equation}
    \label{liminf-0}
    \liminf_{t\to\infty}\left(e^{-\underline{\ba}\, t}\|\nabla u(t)\|_{L^2}\right) =0.
\end{equation}
From \eqref{liminf-0} and \eqref{M-Id}, we have
\begin{equation}
    \label{liminf-1}
    \liminf_{t\to\infty}\left(e^{-\underline{\ba}\, t}\| u(t)\|_{H^1}\right) =0.
\end{equation}
The above assertion clearly implies \eqref{exp-scat}. Now, by Sobolev embedding and \eqref{observ}, we infer that
\begin{eqnarray*}
\|U_{\ba}(t)u(T)\|_{L^\theta(T,\infty;~L^{r_0})} &\lesssim& \left\|e^{-\baa(t)} \|u(T)\|_{H^1}\right\|_{L^\theta(T,\infty)} \\
&\lesssim& e^{-\underline{\ba}\, T} \,\|u(T)\|_{H^1} 
\lesssim \varepsilon.  
\end{eqnarray*}
Therefore, the assumption of Proposition \ref{Oh-To} is fulfilled at the initial time $t=T$, and its conclusion holds true.

This achieves the proof of Theorem \ref{scat3-bis}.

\section{Conclusion and Open Problems}

In this work, we have examined the long-time behavior of solutions to the focusing nonlinear Schr\"odinger equation with a time-dependent damping term in the energy sub-critical regime. Under non-classical assumptions on the damping term, we have obtained some scattering results in the energy space.

Our work extends previous results on the constant damping case to the more general setting of time-dependent damping, providing new insights into the interplay between the damping term and the initial data.

More precisely, for small initial data in the sense of \eqref{small-new}, global solutions exist and scatter in $H^{1}(\mathbb{R}^{N})$ for the intercritical regime. Additionally, we showed that if the solution remains bounded in $L^{p+1}$, then it scatters in $H^{1}(\mathbb{R}^{N})$. Our results highlight the importance of the damping term in ensuring global existence and scattering as the condition $\underline{\mathbf{a} }> 0$ plays a crucial role in guaranteeing the decay of the solution.

Despite the significant progress achieved in this study, several  open questions  remain unresolved. Our results require $\underline{\mathbf{a}} > 0$ for scattering. It would be interesting to investigate whether scattering can still occur when $\underline{\mathbf{a}} = 0$, particularly for damping terms that decay slowly in time, such as in \eqref{Main-New}. A recent progress in this direction was done in \cite{Bamri, Tayachi-23}. The general case $a, \theta>0$ in \eqref{Main-New} remains largely unexplored. We aim in \cite{HM-New} to provide some insights on the scattering behavior for \eqref{Main-New}.

Furthermore, the case $\underline{\mathbf{a}} < 0$ remains an intriguing open question to be investigated. Nevertheless, as discussed in Remark \ref{Rem2.3} $(iv)$, we conjecture that the assumption $\underline{\mathbf{a}} > 0$ can be relaxed in some particular settings, requiring only that $\mathbf{a}(t) > 0$ holds for sufficiently large times.

Another important direction is to extend our results to the {\it energy-critical case}. The critical regime often exhibits different behavior, and the presence of damping could lead to new phenomena. Some progress in this direction is under investigation in \cite{HMS-New}.

In conclusion, our work contributes to the understanding of the long-time dynamics of the unsteady damped NLS equation. By establishing scattering results under non-classical assumptions, we have extended the existing studies and highlighted the importance of the damping term in controlling the behavior of solutions. 
\vspace{0.5cm}

\hrule

\end{document}